\documentclass{amsart}
\usepackage{latexsym,amsmath,amsthm,verbatim,ifthen,amssymb,graphicx,color,mathrsfs}
\usepackage[all]{xy}
\usepackage{tikz-cd}%carrasquel
\usepackage{color}

\SelectTips{cm}{}

\newtheorem{theorem}{Theorem}[section]
\newtheorem*{theorem*}{Theorem}

 \newtheorem{lemma}[theorem]{Lemma}
 \newtheorem{proposition}[theorem]{Proposition}
 \theoremstyle{definition}
 
 \theoremstyle{remark}
 \newtheorem{remark}[theorem]{Remark}
 
 \numberwithin{equation}{section}

\newcommand{\pa}{\left(}
\newcommand{\pb}{\right)}

\newcommand{\cofib}{\rightarrowtail}

\newcommand{\quism}{\stackrel{\simeq}{\longrightarrow}}

\newcommand{\secat}{{\rm secat}}

\newcommand{\mQ}{\mathbb{Q}}
\newcommand{\mL}{\mathbb{L}}

\begin{document}

\title{Sectional category \`a la Quillen}

\begin{abstract}
In this note we give a characterization of the sectional category of a map between rational spaces in terms of its Koszul-Quillen model.
\end{abstract}

\author[U. Buijs]{Urtzi Buijs}
\address{Departamento de \'Algebra, Geometr\'{\i}a y Topolog\'{\i}a, Universidad
de M\'alaga, Ap. 59, 29080 M\'alaga, Spain}
\email{ubuijs@uma.es}

\author[J. Carrasquel ]{Jos\'e Carrasquel}

\email{jgcarras@gmail.com}

%\thanks{The first was supported by the
% spanish MINECO grant MTM2013-41768-P} \keywords{Homotopy theory; Dynamical systems; Number theory; Algebraic higher structures} \subjclass[2010]{Primary: 55P62;
%Secondary: 54C35}

\maketitle

\section*{Introduction}\label{sect:intro}
The \emph{sectional category} of a continuous map $f\colon X\to Y$ is the least number of open sets of $Y$, minus one, needed to cover $Y$ on each of which $f$ has a local homotopy section \cite{Schwarz66}. This invariant of the weak homotopy type of maps has several particular cases that have been studied for a long time. For instance, the Lusternik-Schnirelmann (LS) category of a path-connected space is the sectional category of its base point inclusion \cite{LS}. Other particular cases include Farber's topological complexity \cite{F} and the LS category of a map \cite{BG,Fox}.\\

If $f$ is a map between rational spaces, then its sectional category depends only on the rational homotopy type of $f$. This simple remark invites the study of sectional category through models. The first step on this direction was the celebrated work of F\'elix and Halperin \cite{FH} where they characterise the LS category of a simply connected rational space of finite $\mQ$-type in terms of its minimal Sullivan model \cite{FHT,S}. This result was generalised in \cite{C} for sectional category of maps between such spaces admitting a homotopy retraction.\\

In this paper, we will give characterisation for the sectional category of a map in terms of its Koszul-Quillen model which we now state.\\

Let $(\mL (Z),\partial)$ be the minimal Quillen model of a space $X$. Then a minimal Quillen model for the cartesian product $X^n$ has the form $(\mL (Z^n),D)$ where $Z^n=\oplus_{j=1}^nZ^n_j$,  

\begin{equation}\label{equa:DefinZi}
	Z^n_j=\bigoplus_{1\le i_1<\cdots<i_j\le n } s^{j-1}(Z_{i_1}\otimes\cdots \otimes Z_{i_j})
\end{equation}	
	 
	 and $Z=Z_i$ \cite[VII.1.(2)]{Tanre}. Observe that in the case that $(\mL (Z),\partial)$ is a cellular Quillen model then so is $(\mL (Z^n), D)$ and the generators $Z^n$ correspond to the product $CW$-structure of $X^{n}$. Observe also that the generators of $Z_n^n$ correspond to products of $n$ positive dimension cells of $X$. 

\begin{theorem*}\label{thm:main}
Let $\varphi\colon(\mL (V),\partial)\cofib (\mL (V\oplus W), \partial)$ be the Koszul-Quillen model of certain continuous map $f\colon X\to Y$  between simply connected rational spaces. Then the sectional category of $f$ is the smallest $n$ for which there exists a dgl map \[\alpha\colon (\mL(V\oplus W ),\partial)\to (\mL((V\oplus W)^{n+1}_1\oplus U^{n+1}),D)\] such that, for all $a\in V\oplus W$, $\alpha(a)=a_1+\cdots +a_{n+1}+\xi$ with \[\xi\in \mL((V\oplus W)^{n+1}_1)*\mL^+(U^{n+1}).\]

%\begin{center}
%	\begin{tikzcd}
%		&(\mL U^{n+1},D)\ar[dr, "\iota"]&\\
%		(\mL(V\oplus W ),\partial)\ar[ur]\ar[rr, "\delta"']&&(\mL(V\oplus W )^{n+1},D)\\	
%	\end{tikzcd}
%\end{center}

Here $U^{n+1}$ denotes the quotient vector space $(V\oplus W)^{n+1}_{\ge 2}$ by $s^n(W_1\otimes\cdots\otimes W_{n+1})$ and $a_i$ means the inclusion of $a$ into $V_i\oplus W_i$. 

\end{theorem*}

Observe that $\alpha$ is in fact a model for the diagonal inclusion which does not \emph{touch} the top cells.
\vskip .5cm
The previous Theorem and the technical Lemmas included in Section 1 of this work constitute the theoretical background necessary to develop the project consisting of studying results in Rational Homotopy Theory related to Cartesian products of spaces.
For this purpose, in future works it will be necessary to describe explicitly the Quillen models of Cartesian products used in this article  as well as a model for the diagonal map.

In view of the characterization of the sectional category used in this article, these explicit models seem essential to understand invariants such as the LS category or topological complexity from the point of view of rational homotopy using Quillen models.

\section{Preliminaries}
The sectional category of an inclusion of a sub CW-complex $f\colon X\hookrightarrow Y$ can be characterize using the \emph{fat wedge} \[T^n(i)=\left\{ (y_1,\ldots, y_n)\in Y^n\colon \ y_l\in f(X) \mbox{ for at least one } l\in \{1,\cdots n\}\right\}.\]

\begin{proposition}\cite{Schwarz66}\label{prop:schwarz}
	The sectional category of $f$ is the smallest $n$ for which the diagonal $Y\to Y^{n+1}$ factors up to homotopy through the fat wedge $T^{n+1}(f)$.
\end{proposition}

Observe that the fat wedge can be written as \[T^{n+1}(f)=\bigcup_{i=0}^nY^i\times X\times Y^{n-i},\] therefore, considering the product CW-structure, it can be seen as a sub CW-complex of $Y^{n+1}$ where the cells of $Y^{n+1}/T^{n+1}(i)$ are the products of $n+1$ cells of $Y/X$.\\

Using the notation of (\ref{equa:DefinZi}), the description of the model of the cartesian product of \cite[VII.1.(2)]{Tanre} can be modified to get

\begin{proposition}\label{prop:mainModel}
	Let $X$ be a topological space and $\mL(Z,\partial)$ be a minimal Quillen model for $X$. Then there is a model for the diagonal inlusion $X\hookrightarrow X^n$,
	
	\begin{center}
		\begin{tikzcd}
		&\mL(Z^n_1\oplus\cdots \oplus Z^n_n, D)\ar[d, "\varphi", "\simeq"']\\
		\mL(Z,\partial)\ar[r, "\Delta"']\ar[ur, "\delta"]&\mL(Z_1,\partial)\times\cdots\times \mL(Z_n,\partial)\\
		\end{tikzcd}
	\end{center}
	
	\begin{itemize}
		\item [(i)] $D(z_i)=\partial z_i$ for $z_i\in Z_i$,
		\item [(ii)] $\Delta$ is the diagonal $n$ inclusion of dgls,
		\item [(iii)] $\varphi(Z^n_2\oplus\cdots Z^n_n)=0$ and $\varphi(z_i)=(0,\ldots,0,z,0,\ldots,0)$ in $i$th position,
		\item [(iv)] for all $s(z_i\otimes z_j)\in Z^n_2$, $D(s(z_i\otimes z_j))=[z_i,z_j]+\xi$ with $\xi\in \mL(Z^n_1)*\mL^+(Z^n_2)$,
		\item[(v)] for $i\ge 3$, $D(Z^n_i)\subset (\mL(Z^n_{1}\oplus \cdots Z^n_{i-2})*\mL^+(Z^n_{i-1}\oplus Z^n_{i})$,
		\item[(vi)] for all $z\in Z$, $\delta(z)=z_1+\cdots+z_n+\xi$ with $\xi\in \mL (Z^n_1)*\mL^+(Z^n_2\oplus\cdots Z^n_n)$.
	\end{itemize}
\end{proposition}

One part of the proof comes from induction and the following modification of \cite[VII.1.(2)]{Tanre}.

\begin{proposition}\label{prop:ModelFor2}
	Let $(\mL (V),\partial)$ and $(\mL (W),\partial)$ be minimal cellular Quillen models for $X$ and $Y$ respectively. Then the minimal cellular Quillen model for $X\times Y$ has the form $\varphi\colon \left(\mL (V\oplus W\oplus s(V\otimes W)), D\right)\quism (\mL V,\partial)\times (\mL W,\partial)$, with $\varphi(v)=v$, $\varphi(w)=w$, $\varphi(s(v\otimes w))=0$, $D(v)=\partial(v)$, $D(w)=\partial(w)$ and \[D(s(v\otimes w))=[v,w]+D^+(s(v\otimes w)),\] where $D^+(s(v\otimes w))\in I_s:=\mL (V\oplus W)*\mL^+(s(V\otimes W))$.
\end{proposition}

In the proof of Proposition \ref{prop:ModelFor2} we will constantly use the following results.

\begin{lemma}\label{lemma:Quism}
	Let $L:=\left(\mL (V\oplus W\oplus s(V\otimes W)), D\right)$ be as in Proposition \ref{prop:ModelFor2} and $\varphi\colon L\rightarrow (\mL (V),\partial)\times (\mL (W),\partial)$ defined, for $v\in V$ and $w\in W$, as $\varphi(v)=v$, $\varphi(w)=w$ and $\varphi(s(v\otimes w))=0$. Then $\varphi$ is a quasi-isomorphism.
\end{lemma}
\begin{proof}
	We will use a spectral sequence argument. Define an increasing filtration on $L$ as $F(k)_r=\oplus_{p=0}^k\pa\mL(V\oplus W)*\mL^{r-p}s(V\otimes W)\pb_r$, where $r$ denotes degree. Observe that $D(F(k))\subset F(k)$ because $s(V\otimes W)$ is concentrated in positive degrees. Filter also $\mL (V)\times \mL (W)$ by degree, that is $G(k)=\pa\mL (V)\times \mL (W)\pb_{\le k}$. Observe that $\varphi$ respects these filtrations. Indeed, if $x\in\mL(V\oplus W)*\mL^{|x|-k}s(V\otimes W)$ then either $\varphi(x)=0$ or $|x|=k$. Then the induced morphism is \[E^0(\varphi)\colon \left(\mL (V\oplus W\oplus s(V\otimes W)), D_0\right)\rightarrow (\mL (V),0)\times (\mL (W),0)\]
	with $D_0(v)=D_0(w)=0$ and $D(s(v\otimes w))=[v,w]$, for $v\in V$ and $w\in W$. By \cite[VII.7. (10)]{Tanre}, $E^0(\varphi)$ is a quasi-isomorphism.
\end{proof}

We also need
\begin{lemma}\label{lemma:Existence}
	Let $\left(\mL (V\oplus W\oplus s(V\otimes W)), D\right)$ be as in Proposition \ref{prop:ModelFor2}. Then for every $\xi\in \mL^+(V)*\mL^+(W)$ there exist $\beta(\xi),D^+(\beta(\xi))\in I_S$ such that $D(\beta(\xi))=\xi+D^+(\beta(\xi))$.
\end{lemma}

\begin{proof}
	The result is obvious if $\xi\in V*W$. Now suppose the result is true for elements of bracket length less than $n$ and let $\xi\in \mL^+(V)*\mL^+(W)$ be of bracket length $n\ge 3$. By Jacobi, $\xi$ is linear combination of elements of the form $[\xi',\omega]$ with $\xi'\in \mL^+(V)*\mL^+(W)$ of bracket length less than $n$. Then, by induction \[D([\beta(\xi'),\omega])=[\xi',\omega]+[D^+(\beta(\xi'),\omega]-(-1)^{|\xi'|}[\beta(\xi'),\partial(w)]\] with $[D^+(\beta(\xi'),\omega]-(-1)^{|\xi'|}[\beta(\xi'),\partial(w)]\in I_s$.
\end{proof}

We may now proceed to the 

\begin{proof}[Proof of Proposition \ref{prop:ModelFor2}]
First write $V=\oplus_{m\ge 0}V_m$ with $\partial V_m\subset \mL (V_{<m})$ and analogously for $W$. Now write  $L_m:=\left(\mL (V\oplus W_{<m}\oplus s(V\otimes W_{<m})), D\right)$ and suppose we have \[\varphi_m\colon L_m\quism (\mL (V),\partial)\times (\mL (W_{<m}),\partial)\] as in the statement of the proposition and such that, for every $n$ and $m'<m$, $\pa\mL (V_{\le n}\oplus W_{\le m'}\oplus s(V_{\le n}\otimes W_{\le m'}),D\pb$ is a sub dgl of $L$. Observe that $\varphi_1$ exists by \cite[Section 3]{Lupton-Smith}. We will construct $\varphi_{m+1}$ inductively. Throughout the proof we will denote $L_{n,m}:=\left(\mL (V_{<n}\oplus W_{<m}\oplus s(V_{<n}\otimes W_{<m})), D\right)$, $I_s^{n,m}$ the Lie ideal of $L_{n,m}$ generated by $s(V_{<n}\otimes W_{<m})$ and \[\varphi_{n,m}\colon L_{n,m}\to (\mL (V_{<n}),\partial)\times (\mL (W_{<m}),\partial)\] the obvious restriction.  Observe that, by Lemma \ref{lemma:Quism}, $\varphi_{m,n}$ is a quasi-isomorphism whenever it is defined.\\

Let $v\otimes w\in V_0\otimes W_m$. We have that $\partial([v,w])=(-1)^{|v|}[v,\partial w]$ is a cycle in $\ker\varphi_{1,m}$. Since $\ker\varphi_{1,m}$ is acyclic, there exists $\tau\in \ker\varphi_{1,m}$ such that $D(\tau)=D([v,w])$. Write $\tau=\tau'+\tau''$ with $\tau'\in \mL^+(V)*\mL^+(W_{<m})$ and $\tau''\in I_s$. By Lemma \ref{lemma:Existence} applied to $L_{1,m}$ , there exist $\beta(\tau')\in I_s^{1,m}$ such that $D(\beta(\tau'))=\tau'+D^+(\beta(\tau'))$. Then $D^+(s(v\otimes w)):=D(\beta(\tau'))-\tau\in I_s^{1,m}$ and $D(D^+(s(v\otimes w)))=-D(\tau)=-D([v,w])$. Now define $D(s(v\otimes w))=[v,w]+ D^+(s(v\otimes w))$.\\

By doing this process over homogeneous basis of $V_0$ and $W_m$ we construct a Koszul-Quillen extension $L_m\cofib (L_m*\mL(W_m\oplus s(V_0\otimes W_m),D)$ with $D$ verifying the properties of the proposition's statement and all possible $L_{n,m}$ being sub dgls. In particular, $\varphi_{1,m+1}$ is now defined.\\

Now suppose we have constructed a Koszul-Quillen extension $L_m\cofib (L_m*\mL(W_m\oplus s(V_{<n}\otimes W_m),D)$ with $D$ as desired and all possible $L_{n,m}$ being sub dgls. In particular, $\varphi_{n,m+1}$ and $\varphi_{n+1,m}$ are defined and are quasi-isomorphisms.\\ 

Let $v\otimes w\in V_n\otimes W_m$. We have that $[\partial v,\partial w]$ is a cycle in the acyclic ideal $\ker\varphi_{n,m}$, this implies that there exists $\psi\in\ker\varphi_{n,m}$ such that $D(\psi)=[\partial v,\partial w]$. Now we have a cycle $[\partial v,w]+(-1)^{|v|}\psi$ in $\ker\varphi_{n,m+1}$. As before, there exists an element $\omega\in \ker\varphi_{n,m+1}$, which can actually be taken in $I_s^{n,m+1}$, such that $D(\omega)=[\partial v,w]+(-1)^{|v|}\psi$. Analogously, there exists $\pi\in I_s^{n+1,m}$ such that $D(\pi)=[v,\partial w]-\varphi$. Finally, define $D(s(v\otimes w))=[v,w]-\omega-(-1)^{|v|}\pi\in \ker\varphi_{n+1,m+1}$.\\

By repeating in homogeneous basis, we extend to a Koszul-Quillen extension $L_m\cofib (L_m*(\mL W_m\oplus s(V_{<m+1}\otimes W_m)),D)$ that closes the induction.
\end{proof}

\begin{remark}
	We have proven that such model actually respects \emph{cone-length} filtration in the sense that the restriction to $\pa \mL(V_{\le n}\oplus W _{\le n}\oplus s(V_{\le n}\otimes W_{\le m}) ),D\pb$ is quasi-isomorphic to $(\mL (V_{\le n}),\partial)\times(\mL (W_{\le m}),\partial)$. Moreover we have constructed it such that\[D(s(V_n\otimes W_m))\subset \mL(V_{\le n}\oplus W_{\le m}\oplus s(V_{<n}\otimes W_{\le m})\oplus s(V_{\le n}\otimes W_{< m})).\] Lastly, since the differentials $\partial$ are minimal, $D$ has no linear part.
\end{remark}

To finalize the proof of Proposition \ref{prop:mainModel} we lift $\Delta$ through $\varphi$ to get $\delta$ and then use Lemma \ref{lemma:Existence}, as in the proof of Lemma \ref{prop:ModelFor2}, to \emph{fix} $\delta$ such that it verifies (vi). 

\section{Proof of Theorem}

\noindent Let $\varphi\colon(\mL (V),\partial)\cofib (\mL (V\oplus W), \partial)$ be a Koszul-Quillen model for $f$, then the inclusion $T^{n+1}(i)\hookrightarrow Y^{n+1}$ is modelled by the inclusion \[\iota\colon (\mL ((V\oplus W)^{n+1}_1\oplus U^{n+1}),D)\cofib (\mL (V\oplus W)^{n+1}).\]\\

Now suppose $\secat(f)=n$, then by Proposition \ref{prop:schwarz}, the $m+1$ diagonal factors through the fat-wedge, by standard rational homotopy theory techniques, there is a homotopy commutative diagram 

\begin{center}
	\begin{tikzcd}
		&(\mL ((V\oplus W)^{n+1}_1\oplus U^{n+1}),D)\ar[dr, "\iota"]&\\
		(\mL(V\oplus W ),\partial)\ar[ur, "\alpha"]\ar[rr, "\delta"']&&(\mL(V\oplus W )^{n+1},D)\\	
	\end{tikzcd}
\end{center}

Observe that the $\alpha$ verifies the conditions of Theorem because the linear parts of homotopic morhpisms between minimal dgls coincide.\\

Conversely, suppose a map $\alpha$ of Theorem exists, then we have a commutative diagram
\begin{center}
\begin{tikzcd}
			&\mL(Z^n_1\oplus\cdots \oplus Z^n_n, D)\ar[d, "\varphi", "\simeq"']\\
			\mL(Z,\partial)\ar[r, "\Delta"']\ar[ur, "\iota\circ\alpha"]&\mL(Z_1,\partial)\times\cdots\times \mL(Z_n,\partial)\\
\end{tikzcd} 
\end{center}
which exposes $\iota\circ\alpha$ as a model for the diagonal $n+1$ inclusion. Realizing $\iota\circ \alpha$ gives a homotopy factorization of the diagonal through the fat-wedge.
\vskip .5cm

\noindent\textbf{Acknowledgements.}  
The first author was partially supported by the MICINN grant PID2020-118753GB-I00 of the Spanish
Government and the Junta de Andaluc\'ia grant ProyExcel-00827

This work has been supported by the National Science Centre grant \\ 2016/21/P/ST1/03460 within the European Union's Horizon 2020 research and innovation programme under the Marie Sk\l{}odowska-Curie grant agreement No. 665778.

\begin{flushright}
	\includegraphics[width=38px]{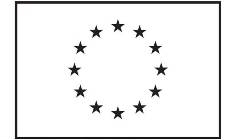}%
\end{flushright}

\end{document}